\documentclass[reqno,b5paper]{amsart}

\usepackage{mathrsfs}
\usepackage{amsmath}
\usepackage{amssymb}
\usepackage{amsthm}
\usepackage{enumerate}
\usepackage[mathscr]{eucal}
\usepackage{eqlist}
\usepackage{graphicx}
\usepackage{hyperref}

\theoremstyle{plain}
\newtheorem{theorem}{Theorem}[section]

\newtheorem{lemma}{Lemma}[section]

\theoremstyle{definition}

\newtheorem{remark}{\textnormal{\textbf{Remark}}}

\numberwithin{equation}{section}

\DeclareMathOperator{\Gal}{Gal}
\newcommand{\Z}{\mathbb{Z}}
\newcommand{\C}{\mathbb{C}}
\newcommand{\mcalO}{\mathcal{O}}
\newcommand{\mcalS}{\mathcal{S}}

\begin{document}
\title[A family of quartic Thue equations over function fields] {ELEMENTARY RESOLUTION OF A FAMILY OF QUARTIC THUE EQUATIONS OVER FUNCTION FIELDS}
\author[CLEMENS FUCHS \and ANA JURASI\'{C} \and ROLAND PAULIN]{CLEMENS FUCHS*$^{,\dag}$ \and ANA JURASI\'{C}** \and ROLAND PAULIN*}

\newcommand{\acr}{\newline\indent}

\address{\llap{*\,}Department of Mathematics\acr
University of Salzburg\acr
Hellbrunnerstr. 34/I\acr
5020 Salzburg\acr
AUSTRIA}
\email{clemens.fuchs@sbg.ac.at, roland.paulin@sbg.ac.at}

\address{\llap{**\,}Department of Mathematics\acr
University of Rijeka\acr
Radmile Matej\v{c}i\'c 2\acr
51000 Rijeka\acr
CROATIA}
\email{ajurasic@math.uniri.hr}

\thanks{$^\dag$Corresponding author: {\sf clemens.fuchs@sbg.ac.at}, phone: +43(662)80445301, fax: +43(662)8044137.}


\begin{abstract}
We consider and completely solve the parametrized family of Thue equations \begin{eqnarray*}X(X-Y)(X+Y)(X-\lambda Y)+Y^4=\xi,\end{eqnarray*} where the solutions $x,y$ come from the ring $\C[T]$, the parameter $\lambda\in\C[T]$ is some non-constant polynomial and $0\neq\xi\in\C$. It is a function field analogue of the family solved by Mignotte, Peth\H{o} and Roth in the integer case. A feature of our proof is that we avoid the use of height bounds by considering a smaller relevant ring for which we can determine the units more easily. Because of this, the proof is short and the arguments are very elementary (in particular compared to previous results on parametrized Thue equations over function fields).
\end{abstract}

\maketitle

\section{Introduction}
Let $R$ be a commutative ring, $R^{\times}$ its group of units and $F \in R[X,Y]$ be a binary irreducible form of degree $d \geq 3$.
An equation $F(X,Y)=m$ with $m \in R^{\times}$ is called a Thue equation since Thue proved in 1909 \cite{thue} the finiteness of solutions $x,y \in R$ of such equations for $R = \Z$.
Nowadays, it is known how to solve algorithmically a Thue equation over a ring $R$ that is finitely generated over $\Z$.
The study of Thue equations over function fields started with Gill's paper \cite{gil}.
In the next 50 years several authors such as Schmidt \cite{sch}, Mason \cite{mas,mas1} and Dvornicich and Zannier \cite{zan} considered the problem to determine effectively all solutions of a given Thue equation over some function field.
In contrast to the number field case $R$ is not necessarily finitely generated over $\Z$, so Thue equations over function fields may have infinitely many solutions.
A criterion for the finiteness of solutions of a given Thue equation was shown by Lettl \cite{letl}.

Since 1990, when Thomas \cite{tom} investigated a parametrized family of cubic Thue equations with positive discriminant, several families of parametrized Thue equations $F_\lambda(X,Y)=m$ have been studied (see a survey  \cite{heub} for further references).
Usually, such a family of equations has finitely many families of solutions depending on the parameter $\lambda$ and finitely many sporadic solutions for certain values of  $\lambda$.
This is however not true in general.
It was shown by Lettl in \cite{letl1} that a family of Thue equations can have sporadic solutions for infinitely many values of the parameter $\lambda$.
The first family of Thue equations over a function field was solved by the first author and Ziegler \cite{fz1}.
In \cite{fz2} they went a step further and solved a family where the parameter itself is a polynomial.
Further results were obtained later by Ziegler in \cite{z0,z}.
We mention that the problem for function fields can be viewed as looking for families of solutions parametrized by polynomials resp.\ algebraic functions and this point of view is behind Lettl's result mentioned above.

In this paper we again consider a family, now with degree $d=4$, where the solutions $x,y$ and the parameter $\lambda$ come from the commutative ring $R=\C[T]$ and the right hand side is a unit in $R$.
In the integer case, this family was considered and completely solved by Mignotte, Peth\H{o} and Roth \cite{mpr}.
We prove:
\begin{theorem} \label{Tm1}
Let $\xi \in \C^{\times}$ and $\lambda \in \C[T] \setminus \C$.
Let
\begin{align*}
F(X,Y) &= X(X-Y)(X+Y)(X - \lambda Y) + Y^4 \\
&= X^4 - \lambda X^3 Y - X^2 Y^2 + \lambda X Y^3 + Y^4 \in \C[T][X,Y].
\end{align*}
Then the set of solutions in $\C[T] \times \C[T]$ of the parametric Thue equation $F(X,Y)$ $ =\xi$ is given by
\[
\mcalS = \{(\zeta, 0), (0, \zeta), (\zeta, \zeta), (-\zeta, \zeta), (\zeta \lambda, \zeta), (-\zeta, \zeta \lambda) ; \, \zeta \in \C^{\times}, \, \zeta^4 = \xi\}.
\]
\end{theorem}

We mention that all the coordinates of the solutions found in Theorem \ref{Tm1} lie in the smaller ring $\C[\lambda]$. This is indeed what one expects in most cases (this phenomenon is part of the notion ``stably solvable'', which has been introduced by Thomas in \cite{tom} for the integer case). An interesting counterexample, however, was given by Lettl in the paper \cite{letl1} as we have already mentioned above.

The family of Thue equations that we consider is a family of splitting type over the ring $\C[T]$, i.e., it has the form
\[
X(X-p_1Y) \cdots (X-p_{d-1}Y)+Y^d=\xi,
\]
where $p_i \in \C[T]$, for $i=1,\dots,d-1$, and $\xi \in \C^{\times}$.
Ziegler \cite{z} proved that such equations have only the trivial solutions
\[
\zeta (1,0), \zeta (0,1), \zeta (p_1,1),\dots, \zeta (p_{d-1},1),
\]
where $\zeta^d=\xi$, if some conditions on the degrees of the polynomials $p_1,\ldots,p_{d-1}$ are satisfied (namely $0<\deg (p_1)<\cdots<\deg (p_{d-1})$ and $1.031d!(d-1)(2d-3)4^{d-1}\deg (p_{d-2})$ $<\deg (p_{d-1})$).
For the equation in Theorem \ref{Tm1}, those conditions are not satisfied, so beside the trivial solutions we also have a non-trivial one, namely $(-\zeta,\zeta\lambda)$ above.

In order to prove Theorem \ref{Tm1} we partially follow the original ideas of Mason \cite{mas}.
In Section 2 we give a decomposition of $F(X,Y)$, and define the relevant ring and determine its unit group.
The decomposition used differs from the simple one suggested by Mason's method, making the relevant ring smaller, and thus making the calculation of the unit group easier.

The main point is that the Galois group of $X^4 - \lambda X^3 - X^2 + \lambda X+ 1$ over $\C[T]$ is special (i.e. not the full symmetric group) for every choice of $\lambda\in\C[T]\backslash\C$; more precisely, the form $X^4 - \lambda X^3 Y - X^2 Y^2 + \lambda X Y^3 + Y^4$ splits into the product of two quadratic forms over a quadratic extension of the ring $\C[T]$. This last extension corresponds to a hyperelliptic curve with two points at infinity (at least if $\deg(\lambda)>2$), and the curve's group of units has rank one over the constants.

In Section 3 we finish the proof of Theorem \ref{Tm1}.
Here we avoid the use of height bounds, hence the whole proof is very elementary.

The special properties of our result are twofold: Firstly, and maybe not so interestingly (or surprisingly), we completely solve an explicit family of degree four whereas before (except for the result with arbitrary degrees from \cite{z}) only families of degree three have been considered. Secondly, and more importantly, we use a completely different method by avoiding the use of heights and instead using the special structure of our equation to carry out the proof. This is in striking contrast to the previous results which all depended on the height (fundamental) inequality found by Mason (cf. \cite{mas1,mas}).

A remark about the notation: throughout the paper, if $A$ is a ring, then $A^2$ denotes the set $\{a^2 ; \, a \in A\}$. Moreover, in the whole paper one can take $\xi=1$, which - to simplify the presentation - we shall assume from this point on.

\section{Determining the unit group}

It is clear that the elements of $\mcalS$ are solutions of the equation $F(X,Y) = 1$. Observe that $F (-X,-Y)=F (X,Y)$ and $F(-Y,X)=F(X,Y)$. Now let $x,y \in \C[T]$ such that $F(x,y) = 1$.
First suppose that $x$ or $y$ is in $\C$.
If $y = 0$, then $x^4 = 1$, so $(x,y) \in \mcalS$.
If $y \in \C^{\times}$, then $x(x - y)(x + y)(x - \lambda y) = 1 - y^4 \in \C$.
The left hand side is either zero or it has positive degree.
Thus $y^4 = 1$ and $x \in \{0, \pm y, \lambda y\}$, so $(x,y) \in \mcalS$.
Suppose that $y \notin \C$ and $x \in \C$.
Then $y(y - x)(y + x)(y + \lambda x) = 1 - x^4 \in \C$, and $\deg(y) > 0$, so $x^4 = 1$ and $y = -\lambda x$, hence $(x,y) \in \mcalS$.
So if $x$ or $y$ is in $\C$, then $(x,y) \in \mcalS$.
We will show that if $x,y \in \C[T]$ and $f(x,y) \in \C^{\times}$, then $x$ or $y$ is in $\C$, thus proving the theorem.

Let $R = \C[T]$, $K = \C(T)$, $\mcalO = R[u]$ and $L = K(u)$, where $u = (\lambda + \sqrt{\lambda^2 - 4})/2$.
Suppose indirectly that $u \in K$.
Using that $R$ is integrally closed and $u^2 - \lambda u + 1 = 0$, we obtain $u \in R$.
Then $f = \sqrt{\lambda^2-4} = 2u - \lambda \in R$, so $4 = \lambda^2 - f^2 = (\lambda + f)(\lambda - f)$.  Hence $\lambda \pm f \in \C$ and therefore $\lambda \in \C$, which is a contradiction.
So $L = K(u) = K(\sqrt{\lambda^2-4})$ is a degree $2$ extension of $K$.
Using the equation $u^2 - \lambda u + 1 = 0$ we obtain that $\mcalO = R \oplus R u$, and that $\mcalO$ is integral over $R$.
The ring $R$ is integrally closed, so $\mcalO \cap K = R$ and $\mcalO^{\times} \cap K = R^{\times} = \C^{\times}$.
We have $u \in \mcalO^{\times}$, because $u^{-1} = (\lambda - \sqrt{\lambda^2 - 4})/{2} = \lambda - u \in \mcalO$.
The field extension $L/K$ is a Galois extension of degree $2$, with Galois group $\Gal(L/K) = \{1, \sigma\}$, where $\sigma(\sqrt{\lambda^2-4}) = -\sqrt{\lambda^2-4}$ and $\sigma(u) = u^{-1}$.
Thus $\sigma|_{\mcalO}$ is an automorphism of the ring $\mcalO$, hence $\sigma|_{\mcalO^{\times}}$ is an automorphism of the group $\mcalO^{\times}$.

In $\mcalO[X,Y]$ we have the decomposition
\[
F(X,Y) = (X^2 - u XY - Y^2)(X^2 - u^{-1} XY - Y^2).
\]
Now $x,y \in R$ and $F(x,y) \in \C^{\times} = R^{\times}$, so $x^2 - y^2 - xy u \in \mcalO^{\times}$.
The following crucial lemma determines the unit group $\mcalO^{\times}$.

\begin{lemma} \label{lemma:unit group}
$\mcalO^{\times} = \{c u^n ; \, c \in \C^{\times}, \, n \in \Z \}$.
\end{lemma}
\begin{proof}
Let $U = \{c u^n ; \, c \in \C^{\times}, \, n \in \Z \}$.
Clearly $U \subseteq \mcalO^{\times}$.
Suppose indirectly that there is a $\theta \in \mcalO^{\times} \setminus U$.
We can write $\theta = a + b u$ for some $a,b \in R$.
If $a = 0$, then $\theta = b u \in \mcalO^{\times}$, hence $b \in \mcalO^{\times} \cap R = R^{\times} = \C^{\times}$, so $\theta \in U$, contradiction.
Therefore $a \neq 0$.
If $b = 0$, then $\theta = a \in \mcalO^{\times} \cap R = R^{\times} = \C^{\times} \subseteq U$, contradiction.
So $b \neq 0$.
Then $\deg (a), \deg (b) \in \Z_{\ge 0}$.
We can choose a $\theta \in \mcalO^{\times} \setminus U$ such that $\deg (b)$ is minimal.
Note that $\theta \in \mcalO^{\times}$, so $\sigma(\theta) = a + b u^{-1} = a + b(\lambda - u) = (a + b \lambda) - bu \in \mcalO^{\times}$, hence
\[
\theta \sigma(\theta) = (a+bu)(a+bu^{-1}) = a^2 + \lambda ab + b^2 \in \mcalO^{\times} \cap K = \C^{\times}.
\]
Using \[
(a+bu) u^{-1} = b + a(\lambda - u) = (b + \lambda a) - a u \in \mcalO^{\times} \setminus U
\]
and
\[
(a+bu) u = au + b(\lambda u - 1) = -b + (a + \lambda b) u \in \mcalO^{\times} \setminus U,
\]
we obtain $\deg(a) \ge \deg(b)$ and $\deg(a + \lambda b) \ge \deg(b)$.
If $\deg(a) > \deg(b)$ or $\deg(a + \lambda b) > \deg(b)$, then $\deg(a(a + \lambda b)) > \deg(b^2) \ge 0$, so
\[
0 = \deg(a^2 + \lambda ab + b^2) = \deg(a(a + \lambda b) + b^2) = \deg(a(a + \lambda b)) > 0,
\]
contradiction.
Thus $\deg(a) = \deg(b) = \deg(a + \lambda b)$. Then $\deg(a^2+ab\lambda+b^2)=2\deg a+\deg\lambda>0$, so the term $a^2+ab\lambda+b^2$ cannot belong to $\C^*$.
So indeed $\mcalO^{\times} = U$.
\end{proof}

\begin{remark}
Observe that the ring $\mcalO$ corresponds to the ring of regular functions on the hyperelliptic curve $S^2=\lambda(T)^2-4$ (at least if $\deg(\lambda)>2$ since, usually, only under this assumption is the curve called hyperelliptic), which has two places at infinity; the divisor of $u$ is supported at these places. The original equation is reduced to the equation $x^2-uxy-y^2=cu^m$, where $u\in\mcalO$ is an explicitly given rational function on the hyperelliptic curve (namely $u=(\lambda(T)+S)/2$), to be solved in $x,y\in\C[T]$.
\end{remark}

\begin{remark}
We could further decompose $F(X,Y)$ as
\[
F(X,Y) = (X - \alpha_1 Y)(X - \alpha_2 Y)(X - \alpha_3 Y)(X - \alpha_4 Y),
\]
where $\alpha_1, \dotsc, \alpha_4$ are elements of a fixed algebraic closure $\overline{K}$ of $K$.
So if $F(x,y) \in \C^{\times}$ for some $x,y \in R$, then $x-\alpha_1 y \in R[\alpha_1]^{\times}$.
Then one would proceed by determining the structure of the unit group $R[\alpha_1]^{\times}$.
However this is probably more difficult to calculate than $\mcalO^{\times} = R[u]^{\times}$, because $K(u)/K$ is a degree $2$ Galois extension, while $K(\alpha_1)/K$ is a degree $4$ non-Galois extension.
\end{remark}

\section{Finishing the proof of Theorem \ref{Tm1}}

Using Lemma \ref{lemma:unit group} we get that $x^2 - y^2 - xyu = c u^m$ for some $c \in \C^{\times}$ and $m \in \Z$.
After multiplying $x,y$ by a nonzero scalar, we may assume that $c = 1$.
Following Mason's approach, one could try to get an upper bound for $|m|$ using height bounds.
Then one could check the finitely many remaining cases one by one.
Here instead we only do the second step, but for a general $m$, which allows us to omit the height bounds.

For every $n \in \Z$ there are unique $A_n, B_n \in R$ such that $u^n = A_n + B_n u$.
Then $x^2-y^2 = A_m$ and $xy = -B_m$.
So $x^2 y^2 = B_m^2$ and
\[
(x^2+y^2)^2 = (x^2-y^2)^2 + 4 x^2 y^2 = A_m^2 + 4B_m^2 = (A_m + 2i B_m)(A_m - 2i B_m) \in R^2.
\]
If $n \in \Z$, then $u^{n+1} = (A_n + B_n u)u = A_n u + B_n(\lambda u - 1) = -B_n + (A_n + \lambda B_n)u$, so $A_{n+1} = -B_n$ and $B_{n+1} = A_n + \lambda B_n$.
In matrix notation $A_n$ and $B_n$ satisfy therefore the following recurrence relation: \[\left(\begin{array}{c}A_{n+1}\\B_{n+1}\end{array}\right)=\left(\begin{array}{cc}0&-1\\1&\lambda\end{array}\right)\left(\begin{array}{c}A_n\\B_n\end{array}\right).\]
The characteristic roots are precisely given by $u,u^{-1}$.
Now clearly, $A_n$ and $B_n$ are polynomials in $\lambda$.
We introduce the sequences $(U_n)_{n \in \Z}$ and $(V_n)_{n \in \Z}$ in $\Z[X]$, defined by the following recursion:
$U_0 = 1$, $V_0 = 0$, and $U_{n+1} = -V_n$ and $V_{n+1} = U_n + X V_n$ for every $n \in \Z$.
Furthermore, let $G_n = U_n + 2i V_n \in \C[X]$ and $\overline{G}_n = U_n - 2i V_n \in \C[X]$ for every $n \in \Z$.
Then $A_n = U_n(\lambda)$ and $B_n = V_n(\lambda)$ for every $n \in \Z$, hence $(U_m^2+4V_m^2)(\lambda) = (G_m \overline{G}_m)(\lambda) \in R^2$.
The following lemma is very useful in this situation.

\begin{lemma} \label{lemma:f(lambda) square => f square}
Let $f \in \C[X] \setminus \{0\}$ and $\lambda \in \C[T] \setminus \C$.
If $\deg(f)$ is even and $f(\lambda) \in \C(T)^2$, then $f \in \C[X]^2$.
\end{lemma}
\begin{proof}
Note that $\C[T]$ is a unique factorization domain such that $\C[T]^{\times} = \C^{\times} \subseteq \C[T]^2$.
So $\C[T] \cap \C(T)^2 = \C[T]^2$, and element of $\C[T] \setminus \{0\}$ is in $\C(T)^2$ if and only if every prime factor has even multiplicity in it.
We can multiply $f$ by a scalar, so we may assume that $f$ is monic.
Then we have a factorization $f(X) = \prod_{j=1}^r (X - \gamma_j)^{m_j}$, where $\gamma_1, \dotsc, \gamma_r$ are pairwise distinct complex numbers, and $m_1, \dotsc, m_r$ are positive integers such that $\sum_{j=1}^r m_j = \deg(f)$ is even.
Let $s$ be the number of $j$'s such that $m_j$ is odd.
Then $0 \le s \le r$ and $s$ is even.
Suppose indirectly that $s \neq 0$.
Then $s \ge 2$.
After reindexing the zeros, we may assume that $m_1, \dotsc, m_s$ are odd and $m_{s+1}, \dotsc, m_r$ are even.
From $f(\lambda) = \prod_{j=1}^r (\lambda - \gamma_j)^{m_j} \in \C[T]^2$ we obtain $\prod_{j=1}^s (\lambda - \gamma_j) \in \C[T]^2$.
If $j \neq j'$, then $\lambda - \gamma_j$ and $\lambda - \gamma_{j'}$ are coprime (because their difference is in $\C^{\times}$).
So $\lambda - \gamma_j \in \C[T]^2$ for every $j \in \{1, \dotsc, s\}$.
Since $s \ge 2$, we get that $\lambda - \gamma_1 = h_1^2$ and $\lambda - \gamma_2 = h_2^2$ for some $h_1, h_2 \in \C[T]$.
Then $(h_1 - h_2)(h_1 + h_2) = h_1^2 - h_2^2 = \gamma_2 - \gamma_1 \in \C^{\times}$, so $h_1 \pm h_2 \in \C^{\times}$. Hence $h_1, h_2 \in \C$ and therefore $\lambda = \gamma_1 + h_1^2 \in \C$, contradiction.
So $s = 0$, hence $f \in \C[X]^2$.
\end{proof}

One can easily show by induction that $\deg(V_n) = |n|-1$ if $n \neq 0$, so $V_n \neq 0$ for $n \neq 0$.
Using this and $U_0 = 1 \neq 0$, we get that $G_n, \overline{G}_n \neq 0$ and $\deg(G_n) = \deg(\overline{G}_n)$ for every $n \in \Z$, hence $U_n^2 + 4 V_n^2 = G_n \overline{G}_n \neq 0$ and $\deg(U_n^2 + 4V_n^2) = 2 \deg(G_n)$.
Then $G_m \overline{G}_m = U_m^2 + 4V_m^2 \in \C[X]^2$ by Lemma \ref{lemma:f(lambda) square => f square}.
From the recursion it is clear that $\gcd(U_n,V_n) = \gcd(U_{n+1}, V_{n+1})$ in $\C[X]$ for every $n \in \Z$, hence $\gcd(U_n,V_n) = \gcd(U_0, V_0) = 1$.
Then $\gcd(G_n, \overline{G}_n) = 1$ too in $\C[X]$, because $U_n = \frac{1}{2}(G_n + \overline{G}_n)$ and $V_n = \frac{1}{4i}(G_n-\overline{G}_n)$.
So $G_m, \overline{G}_m \in \C[X]^2$.

The sequences $(U_n)_{n \in \Z}$, $(V_n)_{n \in \Z}$, $(G_n)_{n \in \Z}$ satisfy the following recursions: $U_{n+2} = X U_{n+1} - U_n$, $V_{n+2} = X V_{n+1} - V_n$, $G_{n+2} = X G_{n+1} - G_n$.
Let
\[
H_n(X) = (-i)^n G_n(iX) \in \C[X],
\]
then $H_0 = 1$, $H_1 = 2$, and $H_{n+2} = X H_{n+1} + H_n$ for every $n \in \Z$.
So $H_n \in \Z[X]$ for every $n \in \Z$, and $H_m \in \C[X]^2$.

Here is a table of $U_n$, $V_n$ and $H_n$ for small values of $n$.
\smallskip
\begin{center}
\scalebox{0.9}{
$\begin{array}{c|c|c|c|c|c|c|c}
n & -3 & -2 & -1 & 0 & 1 & 2 & 3 \\
\hline
U_n & X^3 - 2X & X^2 - 1 & X & 1 & 0 & -1 & -X \\
V_n & -X^2 + 1 & -X & -1 & 0 & 1 & X & X^2-1 \\
H_n & -X^3 + 2X^2 - 2X + 2 & X^2 - 2X + 1 & -X+2 & 1 & 2 & 2X+1 & 2X^2+X+2
\end{array}$}
\end{center}
\smallskip
One can easily show that $H_m \in \C[X]^2 \cap \Z[X]$ implies that $H_m = C h^2$ for some $C \in \Z \setminus \{0\}$ and $h \in \Z[X]$.
We will show that $m \in \{-2, 0, 1\}$.
Suppose indirectly that $m \notin \{-2, 0, 1\}$.
If $n \ge 2$, then $\deg(H_n) = n-1$, and the coefficient of $X^{n-2}$ in $H_n$ is $1$.
If $n \le 0$, then $\deg(H_n) = -n$, and the leading coefficient of $H_n$ is $(-1)^n$.
So $H_n$ is primitive for every $n \in \Z \setminus \{1\}$.
We have assumed that $m \neq 1$, so $H_m = C h^2$ is primitive, hence $C = \pm 1$.
If $m \ge 2$, then the leading coefficient of $H_m = \pm h^2$ is $2$, which is impossible, since $\sqrt{\pm 2} \notin \Z$.
Since $H_{-1} = -X+2 \notin \C[X]^2$, we have $m \neq -1$.
So $m \le -3$.
Then $\deg(H_m) = -m$, so $m$ must be even.
Then $H_m = C h^2$ is monic, so $C = 1$.
One can easily show by induction that $H_{-n}(1) = (-1)^{n+1} F_{n-2}$ for every $n \ge 2$, where $F_l$ denotes the $l$th Fibonacci number.
Note that $F_l > 0$ for $l > 0$, and $-m-2 > 0$, hence $h(1)^2 = H_m(1) = -F_{-m-2} < 0$, contradiction.
So indeed $m \in \{-2, 0, 1\}$.

If $m \in \{0, 1\}$, then $xy = -V_m(\lambda) \in \C$, so $x$ or $y$ is in $\C$.
If $m = -2$, then $x^2-y^2 = U_{-2}(\lambda) = \lambda^2 - 1$ and $xy = -V_{-2}(\lambda) = \lambda$.
Then $x^2-y^2 = x^2y^2 - 1$, so $(x^2+1)(y^2-1) = 0$, hence $x$ or $y$ is in $\C$.
So we have proved that if $F(x,y) = 0$, then $x$ or $y$ is in $\C$, and therefore $(x,y) \in \mcalS$.

\section{Acknowledgements}

The first and third author were supported by a grant of the Austrian Science Fund (FWF): P24574-N26, the second author was supported by the University of Rijeka research grant no. 13.14.1.2.02 and by Croatian Science Foundation under the project no. 6422. The authors are grateful to an anonymous referee for pointing out the geometric situation lying behind our line of reasoning.


\begin{thebibliography}{99}

\bibitem{zan}
DVORNICICH, R.---ZANNIER, U.:
\textit{A Note on Thue's Equation Over Function Fields},
Monatsh. Math. \textbf{118} (1994), 219--230.

\bibitem{fz1}
FUCHS, C.---ZIEGLER, V.:
\textit{On a Family of Thue Equations over Function Fields},
 Monatsh. Math. \textbf{147} (2006), 11--23.

\bibitem{fz2}
FUCHS, C.---ZIEGLER, V.:
\textit{Thomas' Family of Thue equations over Function Fields},
Quart. J. Math. Oxford \textbf{57} (2006), 81--91.

\bibitem{gil}
GILL,  B. P.:
\textit{An analogue for algebraic functions of the Thue-Siegel theorem},
Ann. of Math. Ser. 2 \textbf{31} (1930), 207--218.

\bibitem{heub}
HEUBERGER, ~C.:
\textit{Parametrized Thue Equations - A Survey}.
Proceedings of the RIMS symposium "Analytic Number Theory and Surrounding Areas", Kyoto, Oct 18-22, 2004,
RIMS K\^{o}ky\^{u}roku 1511, August 2006, pp.~82--91.

\bibitem{letl1}
LETTL, G.:
\textit{Parametrized solutions of Diophantine equations},
Math. Slovaca \textbf{54} (2004), 465--471.

\bibitem{letl}
LETTL, G.:
\textit{Thue equations over algebraic function fields},
Acta Arith. \textbf{117} (2005), 107--123.

\bibitem{mas1}
MASON, R. C.:
\textit{On Thue's equation over function fields},
J. London Math. Soc. (2) \textbf {24(3)} (1981), 414--426.

\bibitem{mas}
MASON, R. C.:
\textit{Diophantine Equations over Function Fields},
Cambridge University Press, Cambridge, 1984.

\bibitem{mpr}
MIGNOTTE, M.---PETH\H{O},A.---ROTH, R.:
\textit{Complete solutions of quartic Thue and index form equations},
Math. Comp. \textbf{65} (1996), 341--354.

\bibitem{sch}
SCHMIDT, W. M.:
\textit{Thue's equation over function fields},
J. Austral. Math. Soc. Ser. A \textbf{25} (1978), 385--422.

\bibitem{tom}
THOMAS, E.:
\textit{Complete solutions to a family of cubic diophantine equations},
J. Number Theory \textbf{34(2)} (1990), 235--250.

\bibitem{thue}
THUE, A.:
\textit{\"{U}ber Ann\"{a}herungswerte algebraischer Zahlen},
J. Reine Angew. Math. \textbf{135} (1909), 284--305.

\bibitem{z}
ZIEGLER, V.:
\textit{On Thue equations of splitting type over function fields},
Rocky Mountain Journal of Mathematics \textbf{40} (2010), no. 2, 723--747.

\bibitem{z0}
 ZIEGLER, V:
\textit{Thomas' conjecture over Function Fields},
J. Theor. Nombres Bordeaux \textbf{19}  (2007), 289--309.


\end{thebibliography}
\end{document}